\documentclass[a4paper, 11pt]{amsart}
\usepackage[utf8]{inputenc}

\usepackage{latexsym,amssymb,amsmath,graphics}
\usepackage[dvips]{graphicx}

\def\ffi{\varphi}
\def\eps{\varepsilon}
\def\dst{\displaystyle}

\def\R{{\mathbb{R}}}

\def\Z{{\mathbb{Z}}}

\def\d{\,{\mathrm{d}}}
\newcommand{\norm}[1]{{\left\|{#1}\right\|}}
\newcommand{\ent}[1]{{\left[{#1}\right]}}

\newcommand{\scal}[1]{{\left\langle{#1}\right\rangle}}

\newtheorem{lemma}{Lemma}[section]

\newtheorem{theorem}[lemma]{Theorem}

\theoremstyle{definition}
\newtheorem{definition}[lemma]{Definition}

\theoremstyle{remark}

\begin{document}

\title[Sampling on spaces of homogeneous type]{A sampling theorem for functions in Besov spaces on spaces of homogeneous type}

\author[Ph. Jaming \& F. Negreira]{Philippe Jaming \& Felipe Negreira}

\address{Univ. Bordeaux, IMB, UMR 5251, F-33400 Talence, France.
CNRS, IMB, UMR 5251, F-33400 Talence, France.}
\email{Philippe.Jaming@math.u-bordeaux.fr}
\email{Felipe.Negreira@math.u-bordeaux.fr}

\begin{abstract}
In this work we establish a sampling theorem for functions in Besov spaces on spaces of homogeneous type as defined in \cite{HY} in the spirit of their recent counterpart for $\R^d$ established by Jaming-Malinnikova in \cite{JM}. The main tool is the wavelet decomposition presented by Deng-Han in \cite{DH}.
\end{abstract}


\maketitle

\section{Introduction.}

The problem of representing and analyzing functions (signals, images or other data) on settings other than the usual euclidean case has become very active field of research over the past decades. In practice this may be explained by the fact that the meaningful data obtained by an acquisition system is often considered to belong to spaces (manifolds, fractals,...) that differ from
a linear subspaces of $\R^d$. This problem is also of theoretic interest since understanding functions and how to represent them
are intrinsic questions on the ambient space on which they are defined.

Some of the cases treated in recent years comprise the sphere \cite{NPW}, locally compact groups \cite{FG}, or even an arbitrary compact 
manifold \cite{P2}. Here our framework will be the so-called spaces of homogeneous type introduced by Coifman and Weiss in \cite{CW}. 
These spaces include all previous mentioned spaces and many more like
the $d$-sets and $d$-spaces in the sense of Triebel \cite{T,T2} which include various types of fractals,
{\it see e.g.} \cite{DH} for a more complete list of concrete examples.

One key feature in analysis is that, in order to be able to study function spaces $\mathcal{F}(X)$
on some ambient space $X$ on which the functions in $\mathcal{F}(X)$ are defined,
one needs a proper representing system $\{\ffi_j\}_j$
(orthonormal bases, frames,...).
Once such a system is available, one associates to any $f\in \mathcal{F}(X)$ 
its coefficients in the representing system $\{\scal{f,\psi_j}\}_j$ (where $\psi_j$ is a ``dual system'')
and then characterizes the fact that 
$f\in \mathcal{F}(X)$ by the behavior of those coefficients. Once this is done, one hopes to be able to
reconstruct $f$ from the coefficients via a summation $f\equiv\sum \scal{f,\psi_j}\ffi_j$.
A key feature in this paper is that such decompositions are available on functions spaces over
spaces of homogeneous type and, moreover, the representation system shares many aspects of the wavelets on $\R^d$
({\it see e.g.} \cite{DH,HX,HY}).

However, one now faces a practical issue. A typical measurement system would not provide the coefficients
$\scal{f,\psi_j}$ but rather values of $f$ at some points of the ambient space $X$.
The aim of Sampling Theory is precisely to reconstruct a function from its samples. The most famous result in that direction is the 
Shannon-Whittaker-Kotelnikov Theorem which states that a band-limited function on $\R$ can be reconstructed from its regular samples.
This theorem has been extended in numerous ways ({\it see e.g.} \cite{Un,Za} and references therein).
To some extend the work of Pesenson \cite{P1,P2} or F\"uhr-Gr\"ochenig \cite{FG} is an adaptation of this classical result
to other ambient spaces. The general idea of a sampling theorem can be stated as follows: if a function has moderate oscillation, then it is well approximated in the neighborhood of a point by its value at that point. If those neighborhoods cover sufficiently well the space, the sample should allow to describe the function globally. Viewed this way, a sampling theorem is roughly characterized by two parameters: the sampling rate and the oscillation of the function.

While the sampling rate is somehow conditioned by the geometry of the ambient space, oscillation can be measured by different means. In that sense, the first author and E. Malinnikova \cite{JM} proved a sampling theorem on $\R^d$ in which, instead of using Paley-Wiener spaces, oscillation is measured in the scale of Besov spaces $B^s_{p,q}$. In a preliminary report to this work \cite{JN} we have used this techniques with the representation system presented in \cite{NPW} to give a proof of this result in the $d$-dimensional sphere. This work could be extended to any compact Riemannian manifold by replacing \cite{NPW} with its extension by Geller-Pesenson \cite{GP}. However, our proof here allows for a further generalization to spaces of homogeneous type.

\medskip

\noindent{\bf Main Theorem.}
{\sl Let $X$ be a space of homogeneous type and $1\leqslant p\leqslant\infty$. Then, given $0<\eps<1$ and $K>0$, there exists sets of points $\{a_n\}_n$ on $X$ and constants $c_1,c_2$ such that 
\begin{equation*}
c_1(1-\eps)\norm{f}_{L^p}\leqslant\left(\sum_n|f(a_n)|^p\right)^{1/p}\leqslant c_2(1+\eps)\norm{f}_{L^p}.
\end{equation*}
holds true for all $f\in B^{d/p}_{p,1}(X)$ with $\norm{f}_{B^{d/p}_{p,1}}\leqslant K\norm{f}_{L^p}$. 
}

\medskip

The actual result is more precise, {\it see} below. The proof consists in using the wavelets constructed by Deng and Han in \cite{DH}
and to define the Besov spaces in terms of the wavelet coefficients in the spirit of the characterization of
Besov spaces on $\R^d$ proved by Y. Meyer \cite{M}. This characterization corresponds to usual Besov spaces in most cases.
The proof than consists in carefully adapting the proof from \cite{JM} to the geometry of homogeneous spaces.

\smallskip

Let us now breifly describe the content of this article. In section 2 we make a quick review of spaces of homogeneous type as defined in \cite{HY}. In section 3 we show some useful proprieties of wavelet family constructed in \cite{DH} as well as the characterization of Besov spaces in this context. Finally, in section 4 we prove our main result.

\section{General framework.}

We begin by describing the general framework in which we are going to work. A quasi-metric $\rho$ in a set $X$ is a function $\rho:X\times X\rightarrow[0,\infty)$ satisfying
\begin{enumerate}\renewcommand{\labelenumi}{(\roman{enumi})}
\item $\rho(x,y)=0$	if and only if $x=y$,

\item $\rho(x,y)=\rho(y,x)$ for all $x,y\in X$,

\item there exists a constant $A>0$ such that for all $x,y,z\in X$
\begin{equation}\label{quasi-metric}
\rho(x,y)\leqslant A(\rho(x,z)+\rho(z,y)).
\end{equation}
\end{enumerate}

Following \cite{HY}, a \textit{space of homogeneous type} $(X,\rho,\mu)_{d,\theta}$ is a set $X$ endowed with a quasi-metric $\rho$ and a non-negative Borelian measure $\mu$, where $d>0$ and $0<\theta\leqslant1$ are such that there exist a constant $C>0$ for which given any $0<r<\text{diam}(X)$ and any $x,x',y\in X$
\begin{gather}
\label{dimension}C^{-1}r^d\leqslant\mu(B(x,r))\leqslant Cr^d,
\\
\label{lipschitz}|\rho(x,y)-\rho(x',y)|\leqslant C\rho(x,x')^\theta(\rho(x,y)+\rho(x',y))^{1-\theta}.
\end{gather}
Macias and Segovia \cite{MS} have proved that these spaces are just the spaces of homogeneous type in the sense of Coifman and Weiss \cite{CW}, whose definitions only require $\rho$ to be a quasi-metric, without \eqref{lipschitz}, and $\mu$ to satisfy the doubling condition, which is weaker than \eqref{dimension}.

The parameter $d$ in \eqref{dimension} is sometimes called the \textit{homogeneous dimension} of $X$. And, indeed, it is clear from the examples given in the introduction that $d$ refers to some kind of dimension of the space. In addition, $\theta$ in \eqref{lipschitz} measures how close the space is of being metric, i.e. when $\rho$ is a metric then we can take $\theta=1$. In this case we get what is known as an Ahlfors $d$-regular space. Nonetheless it is important to remark that in some cases the parameter $\theta$ can not be omitted, as for instance the case $\R^d$ endowed with a non-isotropic metric and the Lebesgue measure.

Let us end this section by mentioning a useful fact which follows from \eqref{dimension}: for $\alpha>-d$ and $r>0$
\begin{equation}\label{useful}
\int_{B(x,r)}\rho(z,x)^{\alpha}\,\d\mu(z)\approx r^{\alpha+ d}
\end{equation}
where the constants depend only on the ambient space $X$.

\section{Wavelet expansion.}

Wavelets in arbitrary $X$ are defined through what is called an \textit{approximation to the identity} 
(see \cite[$\S1$]{HY}). However to define any wavelet system, as in $\R^d$, at some point 
we will need to partition our ambient space in a uniform manner for each scale $j\in\Z$. This can be done in different ways according to the space (e.g. for the sphere there is construction with spherical simplices \cite{MNW}, an arbitrary compact manifold it is decomposable into a $r$-lattice \cite{P2}, ...), but here we will take a more general approach by using the following result of Hytönen and Kairema.

\begin{theorem}[{\cite[Theorem 2.1]{HK}}]
For each $j\in\Z$ there exist a countable collection of open subsets $\{Q^j_k\}_{k\in I_j}$ in such that
\begin{enumerate}\renewcommand{\labelenumi}{(\roman{enumi})}
\item for every $j\in\Z$,
\begin{equation*}
X=\bigcup_{k\in I_j}Q^j_k,
\end{equation*}
		
\item there are constants $r_0,r_1>0$ for which given any pair $(j,k)$ there exist at least one $y^j_k\in Q^j_k$ with
\begin{equation}\label{eps}
B(y^j_k,r_02^{-j})\subset Q^j_k\subset B(y^j_k,r_12^{-j}),
\end{equation}
		
\item if $i\geqslant j$ then
\begin{equation}\label{disjoint}
\text{either }\:Q^i_{k'}\subset Q^j_k\:\text{ or }\:Q^i_{k'}\cap Q^j_k=\emptyset,
\end{equation}
		
\item for each $(j,k)$ and $i<j$ there exist a unique $k'$ for which
\begin{equation*}
Q^j_k\subset Q^i_{k'}.
\end{equation*}
\end{enumerate}
\end{theorem}

We will call the sets $Q^j_k$ \textit{dyadic cubes} and we will refer to the points $y^j_k$ in \eqref{eps} as their \textit{centers}.
Note that \eqref{eps} implies that $\mu(Q_k^j)\approx2^{-jd}$ with constants that depend on $r_0,r_1$ and $X$.
These cubes are used to construct a wavelet type system that is localized around their respective centers. We will need the following result that can be found {\it e.g.} in \cite[$\S3.5$]{DH}.

\begin{theorem}
Let $\{Q^j_k\}_{k\in I_j}$ be the family of dyadic cubes of the previous theorem and $y^j_k$ their respective centers. There exists two families of functions $\{\ffi_{j,k}\}_{j,k},\{\psi_{j,k}\}_{j,k}$ and a constant $C_\ffi>0$ such that

--- each $\{\ffi_{j,k}\}_{j,k}$ satisfies a size condition
\begin{equation}\label{wavelet-size}
\ffi_{j,k}(x)=0\:\text{ if }\:\rho(x,y^j_k)\geqslant C_\ffi 2^{-j}\quad\text{ and }\quad\norm{\ffi_{j,k}}_\infty\leqslant C_\ffi 2^{jd/2},
\end{equation}

--- each $\{\ffi_{j,k}\}_{j,k}$ satisfies a smoothness condition
\begin{equation}\label{wavelet-smooth}
|\ffi_{j,k}(x)-\ffi_{j,k}(y)|\leqslant C_\ffi 2^{j(d/2+\theta)}\rho(x,y)^\theta\quad\text{ for all }\:x,y\in X.
\end{equation}

--- the families are dual in the sense that 
\begin{equation}\label{decomposition}
f=\sum_{j\in\Z}\sum_{k\in I_j}\langle f,\psi_{j,k}\rangle\ffi_{j,k}
\end{equation}
holds true for any $f\in L^2(X)$.
\end{theorem}

Unlike wavelet systems in $\R^d$, here the family $\{\ffi_{j,k}\}_{j,k}$ does not necessarily constitute an orthogonal 
system.\footnote{An orthonormal wavelet system on spaces of homogeneous type has been constructed in \cite{AH}.
The construction is much more involved than the wavelet frames considered here and does not bring any significant improvement
in our results.}
 In particular their supports are not necessarily disjoint. However, the size condition \eqref{wavelet-size} together with \eqref{eps} and \eqref{disjoint} imply that there exist a constant $N>0$ for which given any $j\in\Z$ the supports of $\{\ffi_{j,k}\}_{k\in I_j}$ have finite multiplicity $N$.

We refrain from listing properties of the family $\{\psi_{j,k}\}_{j,k}$ as they will not
be used here except for the fact that $\langle f,\psi_{j,k}\rangle$ makes sense for $f\in L^1_{loc}$. 
This allows us to define the Besov spaces on $X$ in the following way:

\begin{definition}
Let $\{\psi_{j,k}\}_{j,k}$ be the dual family of \eqref{decomposition}. Then, given $0<p,q\leqslant\infty$ and $s\in\R$, the Besov space $B^s_{p,q}(X)$ is defined as the set of all functions $f\in L^1_{loc}$ such that the norm
\begin{equation*}
\norm{f}_{B^s_{p,q}}:=\left(\sum_{j\in\Z}\left[2^{j\left(s+d\ent{\frac{1}{2}-\frac{1}{p}}\right)}\left(\sum_{k\in I_j}|\langle f,\psi_{j,k}\rangle|^p\right)^{1/p}\right]^q\right)^{1/q}
\end{equation*}
is finite. As usual, the $L^p,\ell^q$ norms are replaced by the sup-norms when $p=\infty$ or $q=\infty$.
\end{definition}

This definition is based on that given by Han, Müller and Yang in \cite{HMY} which follows from Littlewood-Paley theory through the approximations of the identity. 
This is the equivalent definition to that given by Meyer \cite{M} in $\R^d$. It also the way they are presented in the sphere by Narcowich-Petrushev-Ward \cite{NPW}.

A difference characterization of Besov spaces in the general setting of a space of homogeneous type has also been given in e.g. \cite{GKS}, \cite{MY}. This definition coincide to that given by Geller-Pesenson \cite{GP}, and again gives the same Besov space in the Euclidean setting. 

Further Müller and Yang in \cite{MY} proved in the more general case 
that both definitions coincide when $-\theta<s<\theta$ and $\max\{\frac{1}{1+\theta},\frac{1}{1+s+\theta}\}\leqslant p,q\leqslant\infty$.

\section{Sampling on spaces of homogeneous type.}

We are now ready to prove our main theorem.

\begin{theorem}
Let $(X,\rho,\mu)_{d,\theta}$ be a space of homogeneous type, $1\leqslant p\leqslant\infty$, and set $\alpha=\max(1,\frac{d}{p\theta})$, $\beta=\max\left(\frac{p}{d},\frac{1}{\theta}\right)$. For every $l\in\Z$, fix a collection of dyadic cubes $\{Q^l_n\}_{n\in I_l}$ with centers $\{a^l_n\}_n$.

Then, given $0<\eps<1$ and $K>0$, there exist a constant $\kappa=\kappa(p,d,\theta)$ such that if $l\geq\beta\ln\left(\frac{\kappa K}{\eps^\alpha}\right)$
\begin{equation}
\left(\int_X\biggl|f(x)-\sum_{n\in I_l}f(a^l_n)\mathbf{1}_{Q^l_n}(x)\biggr|^p\,\mathrm{d}\mu(x)\right)^{1/p}\leqslant
\eps\norm{f}_{L^p}
\end{equation}
holds true for all $f\in B^{d/p}_{p,1}(X)$ with $\norm{f}_{B^{d/p}_{p,1}}\leqslant K\norm{f}_{L^p}$.
In particular, this implies that
\begin{equation*}
(1-\eps)\norm{f}_{L^p}\leqslant\left(\sum_{n\in I_l}\bigl|f(a^l_n)\mu(Q^l_n)\bigr|^p\right)^{1/p}\leqslant
(1+\eps)\norm{f}_{L^p}
\end{equation*}
whenever $l\geq\beta\ln\left(\frac{\kappa K}{\eps^\alpha}\right)$ and $f\in B^{d/p}_{p,1}(X)$ is such that  $\norm{f}_{B^{d/p}_{p,1}}\leqslant K\norm{f}_{L^p}$.
\end{theorem}

\begin{proof}
Let us first note that
\begin{multline}\label{Lp-ellp}
\left(\int_X\biggl|f(x)-\sum_{n\in I_l}f(a^l_n)\mathbf{1}_{Q^l_n}(x)\biggr|^p\,\mathrm{d}\mu(x)\right)^{1/p}
\\
=\norm{\left(\int_{Q^l_n}|f(x)-f(a^l_n)|^p\,\d\mu(x)\right)^{1/p}}_{\ell^p_{I_l}}
\end{multline}
and so the $L^p$-norm of $\displaystyle f(x)-\sum_{n\in I_l}f(a^l_n)\mathbf{1}_{Q^l_n}(x)$ can be computed as the $\ell^p$-norm of the sequence $\left\{\norm{f-f(a^l_n)}_{L^p(Q^l_n)}\right\}_{n\in I_l}$.
	
Now, take an arbitrary $n\in I_l$ and consider $x\in Q^l_n$. From \eqref{decomposition} we may write 
\begin{equation}\label{expansion}
f(x)-f(a^l_n)=\sum_{j\in\Z}\sum_{k\in I_j}\langle f,\psi_{j,k}\rangle\bigl(\ffi_{j,k}(x)-\ffi_{j,k}(a^l_n)\bigr).
\end{equation}
Recall that by \eqref{wavelet-size} if $\ffi_{j,k}(x)\neq0$ then $\rho(x,y^j_k)\leqslant C_\ffi 2^{-j}$.
We therefore introduce
\begin{equation*}
I_j^x:=\{k\in I_j:\rho(x,y^j_k)\leqslant C_\ffi 2^{-j}\},\quad I_j^{n,x}:=I_j^{a^l_n}\cup I_j^x,\quad I_j^n:=\bigcup_{x\in Q^l_n}I_j^x.
\end{equation*}

First note that if $k\in I_j^x$ and $z\in Q_k^j$, then from \eqref{quasi-metric} we get
\begin{equation*}
\rho(x,z)\leqslant A\bigl(\rho(z,y_k^j)+\rho(x,y_k^j)\bigr)\leqslant A(C_\ffi+r_1)2^{-j}.
\end{equation*}
Therefore $Q_k^j\subset B(x,C2^{-j})$ with $C=A(C_\ffi+r_1)$. But since the cubes $Q_k^j$'s are disjoint of volume $\approx2^{-jd}$ a measure counting argument shows that $2^{-jd}|I_j^x|\lesssim 2^{-jd}$, and thus $I_j^x$ is a finite set with $\#I_j^x\leqslant R$ where $R$ is a constant that depends only on $X$ and the dyadic decomposition.

Further, let us introduce $E_j^n(f):=\left(\sum_{k\in I_j^r}|\langle f,\psi_{j,k}\rangle|^p\right)^{1/p}$. From \eqref{expansion} we deduce that
\begin{multline}\label{wavelet-estimate}
|f(x)-f(a^l_n)|\leqslant
\sum_{j\in\Z}\sup_{k\in I_j}|\ffi_{j,k}(x)-\ffi_{j,k}(a^l_n)|\sum_{k\in I_j^{n,x}}|\langle f,\psi_{j,k}\rangle|
\\
\leqslant R^{1/p'}\sum_{j\in\Z}\sup_{k\in I_j}|\ffi_{j,k}(x)-\ffi_{j,k}(a^l_n)|E_j^n(f)
\end{multline}
with Hölder's inequality.
	
From \eqref{wavelet-size} and \eqref{wavelet-smooth} we know that
\begin{equation*}
|\ffi_{j,k}(x)-\ffi_{j,k}(a^l_n)|\leqslant
\left\{\begin{array}{l}
2C2^{jd/2},
\\
C2^{jd/2}2^{j\theta}\rho(x,a^l_n)^\theta.
\end{array}\right.
\end{equation*}
The second inequality improves over the first one when $2^j\rho(x,a^l_n)\lesssim 1$. Then we can split the sum of \eqref{wavelet-estimate} in two parts to obtain
\begin{equation*}
|f(x)-f(a^l_n)|\leqslant C_p\sum_{j\leqslant l}2^{jd/2}2^{j\theta}\rho(x,a^l_n)^\theta E_j^n(f)+C_p\sum_{j>l}2^{jd/2}E_j^n(f).
\end{equation*}
	
Next, taking the $L^p$-norm over $Q^l_n$ and using the triangular inequality, we get
\begin{multline*}
\left(\int_{Q^l_n}|f(x)-f(a^l_n)|^p\,\d\mu(x)\right)^{1/p}
\\
\leqslant C_p\sum_{j\leqslant l}2^{jd/2}2^{j\theta}\left(\int_{Q^l_n}\rho(x,a^l_n)^{\theta p}\,\d\mu(x)\right)^{1/p}E_j^n(f)
\\
+C_p\sum_{j>l}2^{jd/2}\mu(Q^l_n)^{1/p}E_j^n(f)
\\
\leqslant C_p\sum_{j\leqslant l}2^{jd/2}2^{j\theta}2^{-l(\theta+d/p)}E_j^n(f)
+C_p\sum_{j>l}2^{jd/2}2^{-ld/p}E_j^n(f)
\end{multline*}
where we used \eqref{useful} together the fact that $\mu(Q^l_n)\approx2^{-ld}$ in the last inequality. So when we take the $\ell^p$-norm over $I_l$ we have that
\begin{multline}\label{ellp}
\norm{\left(\int_{Q^l_n}|f(x)-f(a^l_n)|^p\,\d\mu(x)\right)^{1/p}}_{\ell^p_{I_l}}
\\
\leqslant C_p\sum_{j\leqslant l}2^{jd/2}2^{j\theta}2^{-l(\theta+d/p)}\norm{E_j^n(f)}_{\ell^p_{I_l}}
\\
+C_p\sum_{j>l}2^{jd/2}2^{-ld/p}\norm{E_j^n(f)}_{\ell^p_{I_l}}.
\end{multline}

To estimate the $\ell^p$ norm of $E_j^n(f)$ we write for each $k\in I_j$, $\Lambda_{j,l}^k:=\{n\in I_l:k\in I_j^n\}$, so that
\begin{multline}\label{tau-norm}
\norm{E_j^n(f)}_{\ell^p_{I_l}}=\left(\sum_{n\in I_l}(E_j^n(f))^p\right)^{1/p}
=\left(\sum_{n\in I_l}\sum_{k\in I_j^n}|\langle f,\psi_{j,k}\rangle|^p\right)^{1/p}\\
=\left(\sum_{k\in I_j}\sum_{n\in\Lambda_{j,l}^k}|\langle f,\psi_{j,k}\rangle|^p\right)^{1/p}.
\end{multline}
Now, the same arguments to estimate the cardinal of $I_j^{n,x}$ imply that
\begin{equation*}
\#(\Lambda_{j,l}^k)\leqslant
\left\{\begin{array}{ll}
C'2^{(l-j)d}&\text{if }j\leqslant l,
\\
C'&\text{if }j>l.
\end{array}\right.
\end{equation*}
Introducing this into \eqref{tau-norm} gives us
\begin{equation*}
\norm{E_j^n(f)}_{\ell^p_{I_l}}\leqslant
\left\{\begin{array}{ll}
C'_p2^{(l-j)d/p}\left(\sum_{k\in I_j}|\langle f,\psi_{j,k}\rangle|^p\right)^{1/p}&\text{if }j\leqslant l,
\\
C'_p\left(\sum_{k\in I_j}|\langle f,\psi_{j,k}\rangle|^p\right)^{1/p}&\text{if }j>l,
\end{array}\right.
\end{equation*}
and going back to \eqref{ellp} we obtain that
\begin{multline}\label{first+second}
\norm{\left(\int_{Q^l_n}|f(x)-f(a^l_n)|^p\,\d\mu(x)\right)^{1/p}}_{\ell^p_{I_l}}
\\
\leqslant C_p\sum_{j\leqslant l}2^{jd/2}2^{j(\theta-d/p)-l\theta}\left(\sum_{k\in I_j}|\langle f,\psi_{j,k}\rangle|^p\right)^{1/p}
\\
+C_p\sum_{j>l}2^{jd/2}2^{-ld/p}\left(\sum_{k\in I_j}|\langle f,\psi_{j,k}\rangle|^p\right)^{1/p}=:I+II.
\end{multline}

The second term on the right hand side is simply bounded by
\begin{equation}\label{secondterm}
II\leqslant C_p2^{-ld/p}\norm{f}_{B^{p/d}_{p,1}}.
\end{equation}
As for the first term, we divide the sum over $j\leqslant l$ into two: $j<j_0$ and $j_0\leqslant j\leqslant l$, where $j_0<l$ is to be fixed later. Thus
\begin{align*}
I&=C_p\left(\sum_{j<j_0}+\sum_{j=j_0}^l\right)2^{-l\theta}2^{j(\theta+d/2-d/p)}\left(\sum_{k\in I_j}|\langle f,\psi_{j,k}\rangle|^p\right)^{1/p}
\\
&\leqslant C_p2^{-l\theta}\sum_{j<j_0}2^{j\theta}\norm{f}_{L^p}+C_p2^{-l\theta}\max_{j_0\leqslant j\leqslant l}\left\{2^{j(\theta-d/p)}\right\}\norm{f}_{B^{d/p}_{p,1}}
\\
&=:I_a+I_b
\end{align*}
Running the sum over $j<j_0$ in $I_a$, we get $I_a\leqslant C_p2^{(j_0-l)\theta}\norm{f}_{L^p}$. And for $I_b$ we have
\begin{align*}
I_b&\leqslant C_p2^{-l\theta}\left(2^{j_0(\theta-d/p)}+2^{l(\theta-d/p)}\right)\norm{f}_{B^{d/p}_{p,1}}
\\
&\leqslant C_p2^{-ld/p}\left(2^{(j_0-l)(\theta-d/p)}+1\right)\norm{f}_{B^{d/p}_{p,1}}.
\end{align*}
Altogether we get
\begin{equation}\label{firstterm}
I \leqslant C_p2^{(j_0-l)\theta}\norm{f}_{L^p}+C_p2^{-ld/p}\left(2^{(j_0-l)\left(\theta-d/p\right)}+1\right)\norm{f}_{B^{d/p}_{p,1}}.
\end{equation}

Adding \eqref{secondterm} and \eqref{firstterm} in \eqref{first+second} yields
\begin{multline*}
\norm{\left(\int_{Q^l_n}|f(x)-f(a^l_n)|^p\,\d\mu(x)\right)^{1/p}}_{\ell^p_{I_l}}
\\
\leqslant C_p2^{(j_0-l)\theta}\norm{f}_{L^p}+C_p2^{-ld/p}\left(2^{(j_0-l)\left(\theta-d/p\right)}+2\right)\norm{f}_{B^{d/p}_{p,1}}.
\end{multline*}
But as we saw at the begging in \eqref{Lp-ellp} this is the same to say that
\begin{multline}\label{almostfinished}
\left(\int_X\biggl|f(x)-\sum_{n\in I_l}f(a^l_n)\mathbf{1}_{Q^l_n}(x)\biggr|^p\,\mathrm{d}\mu(x)\right)^{1/p}
\\
\leqslant C_p2^{(j_0-l)\theta}\norm{f}_{L^p}
+C_p2^{-ld/p}\left(2^{(j_0-l)\left(\theta-d/p\right)}+2\right)\norm{f}_{B^{d/p}_{p,1}}.
\end{multline}

We now choose $j_0:=l-\frac{\ln(2C_p/\eps)}{\theta\ln 2}$ so that $C_p2^{(j_0-l)\theta}\leqslant \dst\frac{\eps}{2}$. Then

--- if $p\geqslant d/\theta$, $C_p2^{-ld/p}\left(2^{(j_0-l)\left(\theta-d/p\right)}+2\right)\leqslant 3C_p2^{-ld/p}$. Therefore taking $l\geq \dst\frac{p}{d\ln2}\ln\left(\frac{6C_p\norm{f}_{B^{d/p}_{p,1}}}{\eps\norm{f}_{L^p}}\right)$, \eqref{almostfinished} reduces to
\begin{equation}\label{sampling}
\left(\int_X\biggl|f(x)-\sum_{n\in I_l}f(a^l_n)\mathbf{1}_{Q^l_n}(x)\biggr|^p\,\mathrm{d}\mu(x)\right)^{1/p}\leqslant\eps\norm{f}_{L^p}
\end{equation}
from where the theorem follows.

--- On the other hand, if $p<d/\theta$
\begin{multline*}
C_p2^{-ld/p}\left(2^{(j_0-l)\left(\theta-d/p\right)}+2\right)\leqslant 2C_p2^{-ld/p}2^{(j_0-l)\left(\theta-d/p\right)}
\\
\leqslant 2C_p\max\left(2^{-l\theta},\left(\frac{\eps}{2C_p}\right)^{1-\frac{d}{\theta p}}2^{-ld/p}\right)\leqslant\frac{\kappa}{2}\eps^{1-\frac{d}{p\theta}}2^{-l\theta} 
\end{multline*}
with $\kappa:=4C_p\max\left(1,(2C_p)^{\frac{d}{p\theta}-1}\right)$. Thus, if $l\geq \dst\frac{1}{\theta\ln2}\ln\left(\frac{\kappa\norm{f}_{B^{d/p}_{p,1}}}{\eps^{\frac{d}{p\theta}}\norm{f}_{L^p}}\right)$ we again obtain \eqref{sampling}.
\end{proof}

\section*{Acknowledgments}
The first author kindly acknowledge financial support from the French ANR program, ANR-12-BS01-0001 (Aventures),
the Austrian-French AMADEUS project 35598VB - Char\-ge\-Disq, the French-Tunisian CMCU/UTIQUE project 32701UB Popart.
This study has been carried out with financial support from the French State, managed
by the French National Research Agency (ANR) in the frame of the Investments for
the Future Program IdEx Bordeaux - CPU (ANR-10-IDEX-03-02).  

The second autor is supported by the doctoral grant POS-CFRA-2015-1-125008 of Agencia Nacional de Innovación e Investigación (Uruguay) and Campus France (France).

\end{document}